\documentclass[a4paper,twoside]{article}
\usepackage{a4}
\usepackage{amssymb}
\usepackage{amsmath}
\usepackage{upref}
\usepackage[active]{srcltx}
\usepackage[colorlinks,citecolor=blue,linkcolor=blue]{hyperref}
\usepackage[dvipsnames]{color}
\allowdisplaybreaks[2] 
\newcount\minutes \newcount\hours
\hours=\time \divide\hours 60 \minutes=\hours
 \multiply\minutes-60
\advance\minutes \time
\newcommand{\klockan}{\the\hours:{\ifnum\minutes<10 0\fi}\the\minutes}
\newcommand{\tid}{\today\ \klockan}
\newcommand{\prtid}{\smash{\raise 10mm \hbox{\LaTeX ed \tid}}}
\renewcommand{\prtid}{}
\makeatletter  \headheight 10pt
\def\sectionmark#1{} 
\def\subsectionmark#1{}
\newcommand{\sectnr}{\ifnum \c@secnumdepth >\z@
                 \thesection.\hskip 1em\relax \fi}
\def\@evenhead{\footnotesize\rm\thepage\hfil\leftmark\hfil\llap{\prtid}}
\def\@oddhead{\footnotesize\rm\rlap{\prtid}\hfil\rightmark\hfil\thepage}
\def\tableofcontents{\section*{Contents} 
 \@starttoc{toc}}
\makeatother
\makeatletter
\def\@biblabel#1{#1.}
\makeatother
\makeatletter
\let\Thebibliography=\thebibliography
\renewcommand{\thebibliography}[1]{\def\@mkboth##1##2{}\Thebibliography{#1}
\addcontentsline{toc}{section}{References}
\frenchspacing 
\setlength{\@topsep}{0pt}
\setlength{\itemsep}{0pt}%
\setlength{\parskip}{0pt plus 2pt}%
} \makeatother
\makeatletter
\def\mdots@{\mathinner.\nonscript\!.%
 \ifx\next,.\else\ifx\next;.\else\ifx\next..\else
 \nonscript\!\mathinner.\fi\fi\fi}
\let\ldots\mdots@
\let\cdots\mdots@
\let\dotso\mdots@
\let\dotsb\mdots@
\let\dotsm\mdots@
\let\dotsc\mdots@
\def\vdots{\vbox{\baselineskip2.8\p@ \lineskiplimit\z@
    \kern6\p@\hbox{.}\hbox{.}\hbox{.}\kern3\p@}}
\def\ddots{\mathinner{\mkern1mu\raise8.6\p@\vbox{\kern7\p@\hbox{.}}%
    \raise5.8\p@\hbox{.}\raise3\p@\hbox{.}\mkern1mu}}
\makeatother
\makeatletter
\let\Enumerate=\enumerate
\renewcommand{\enumerate}{\Enumerate%
\setlength{\@topsep}{0pt}
\setlength{\itemsep}{0pt}%
\setlength{\parskip}{0pt plus 1pt}%
\renewcommand{\theenumi}{\textup{(\alph{enumi})}}%
\renewcommand{\labelenumi}{\theenumi}%
}
\let\endEnumerate=\endenumerate
\renewcommand{\endenumerate}{\endEnumerate\unskip}
\makeatother
\makeatletter
\def\@seccntformat#1{\csname the#1\endcsname.\quad}
\makeatother
\newcommand{\authortitle}[2]{\author{#1}\title{#2}\markboth{#1}{#2}}
\newcommand{\art}[6]{{\sc #1, \rm #2, \it #3 \bf #4 \rm (#5), \mbox{#6}.}}
\newcommand{\auth}[2]{{#1, #2.}}
\newcommand{\idxauth}[2]{{#1, #2.}}
\newcommand{\artin}[3]{{\sc #1, \rm #2,  in #3.}}
\newcommand{\artprep}[3]{{\sc #1, \rm #2, #3.}}
\newcommand{\arttoappear}[3]{{\sc #1, \rm #2, to appear in \it #3}}
\newcommand{\book}[3]{{\sc #1, \it #2, \rm #3.}}
\newcommand{\AND}{{\rm and }}
\RequirePackage{amsthm}
\newtheoremstyle{descriptive}%
  {\topsep}   
  {\topsep}   
  {\rmfamily} 
  {}          
  {\bfseries} 
  {.}         
  { }         
  {}          
\newtheoremstyle{propositional}%
  {\topsep}   
  {\topsep}   
  {\itshape}  
  {}          
  {\bfseries} 
  {.}         
  { }         
  {}          
\theoremstyle{propositional}
\newtheorem{thm}{Theorem}[section]
\newtheorem{prop}[thm]{Proposition}
\newtheorem{lemma}[thm]{Lemma} 
\newtheorem{cor}[thm]{Corollary}
\theoremstyle{descriptive}
\newtheorem{deff}[thm]{Definition}
\newtheorem{example}[thm]{Example}
\newtheorem{remark}[thm]{Remark}
\newtheorem{problem}[thm]{Open problem}
\makeatletter
\renewenvironment{proof}[1][\proofname]{\par
  \pushQED{\qed}%
  \normalfont
  \trivlist
  \item[\hskip\labelsep
        \itshape
    #1\@addpunct{.}]\ignorespaces
}{%
  \popQED\endtrivlist\@endpefalse
} \makeatother
\newcommand{\setm}{\setminus}
\renewcommand{\emptyset}{\varnothing}
\def\vint{\mathop{\mathchoice%
          {\setbox0\hbox{$\displaystyle\intop$}\kern 0.22\wd0%
           \vcenter{\hrule width 0.6\wd0}\kern -0.82\wd0}%
          {\setbox0\hbox{$\textstyle\intop$}\kern 0.2\wd0%
           \vcenter{\hrule width 0.6\wd0}\kern -0.8\wd0}%
          {\setbox0\hbox{$\scriptstyle\intop$}\kern 0.2\wd0%
           \vcenter{\hrule width 0.6\wd0}\kern -0.8\wd0}%
          {\setbox0\hbox{$\scriptscriptstyle\intop$}\kern 0.2\wd0%
           \vcenter{\hrule width 0.6\wd0}\kern -0.8\wd0}}%
          \mathopen{}\int}
\newcommand{\Cp}{{C_p}}
\newcommand{\CpY}{{C_p^Y}}
\newcommand{\CpU}{{C_p^U}}
\DeclareMathOperator{\diam}{diam} 
\DeclareMathOperator{\dist}{dist}
\DeclareMathOperator{\Lip}{Lip}
\DeclareMathOperator{\spt}{supp}
\newcommand{\supp}{\spt}
\newcommand{\loc}{_{\rm loc}}
\newcommand{\al}{\alpha}
\newcommand{\alp}{\alpha}
\newcommand{\dmu}{d\mu}
\newcommand{\de}{\delta}
\newcommand{\eps}{\varepsilon}
\newcommand{\ga}{\gamma}
\newcommand{\Om}{\Omega}
\renewcommand{\phi}{\varphi}
\newcommand{\p}{{$p\mspace{1mu}$}}
\newcommand{\R}{\mathbf{R}}
\newcommand{\Q}{\mathbf{Q}}
\newcommand{\eR}{{\overline{\R}}}
\newcommand{\limplus}{{\mathchoice{\vcenter{\hbox{$\scriptstyle +$}}}
  {\vcenter{\hbox{$\scriptstyle +$}}}
  {\vcenter{\hbox{$\scriptscriptstyle +$}}}
  {\vcenter{\hbox{$\scriptscriptstyle +$}}}
}}
\newcommand{\Np}{N^{1,p}}
\newcommand{\Nploc}{N^{1,p}\loc}
\newcommand{\gat}{{\tilde{\ga}}}
\newcommand{\Gt}{\widetilde{G}}
\newcommand{\Omt}{\widetilde{\Om}}
\newcommand{\ut}{\tilde{u}}
\numberwithin{equation}{section}
\newcommand{\U}{\mathcal{U}}
\newcommand{\YY}{\mathcal{Y}}
\newcommand{\Ll}{\mathcal{L}}
\newcommand{\Ga}{\Gamma}
\newcommand{\Lp}{L^p}
\renewcommand{\P}{\mathcal{P}}
\newcommand{\eqv}{\ensuremath{
\mathchoice{\quad \Longleftrightarrow \quad}{\Leftrightarrow}
                {\Leftrightarrow}{\Leftrightarrow}} }
\newcommand{\imp}{\ensuremath{\Rightarrow} }

\newenvironment{ack}{\medskip{\it Acknowledgement.}}{}

\begin{document}

\authortitle{Anders Bj\"orn, Jana Bj\"orn and Jan Mal\'y}
{Quasiopen and \p-path open sets, and characterizations of quasicontinuity}
\author
{Anders Bj\"orn \\
\it\small Department of Mathematics, Link\"oping University, \\
\it\small SE-581 83 Link\"oping, Sweden\/{\rm ;}
\it \small anders.bjorn@liu.se
\\
\\
Jana Bj\"orn \\
\it\small Department of Mathematics, Link\"oping University, \\
\it\small SE-581 83 Link\"oping, Sweden\/{\rm ;}
\it \small jana.bjorn@liu.se
\\
\\
Jan Mal\'y \\
\it\small Department of Mathematics, Faculty of Science, 
     J. E. Purkyn\v e University,\\
\it\small \v Cesk\'e ml\'ade\v{z}e 8, 
CZ-400 96 \'Ust\'{\i} nad Labem, 
Czech Republic\/{\rm ;} \\
\it \small maly@karlin.mff.cuni.cz
}

\date{}

\maketitle

\noindent{\small
 {\bf Abstract}.
In this paper we give various characterizations
of quasiopen sets and quasicontinuous functions on 
metric spaces.
For complete metric spaces equipped with a doubling measure 
supporting a \p-Poincar\'e inequality
we show that quasiopen and \p-path open sets coincide. 
Under the same assumptions we show that all Newton-Sobolev 
functions on quasiopen sets are quasicontinuous.
}

\bigskip

\noindent {\small \emph{Key words and phrases}:
Analytic set,
characterization,
doubling measure, 
fine potential theory,
metric space, 
Newtonian space,
nonlinear potential theory,
Poincar\'e inequality, 
\p-path open,
quasicontinuous, quasiopen, 
Sobolev space,
Suslin set.
}

\medskip


\section{Introduction}

When studying Sobolev spaces and 
potential theory on an open subset $\Om$  
of $\R^n$ (or of a metric space),
there are two natural Sobolev capacities one can consider.
One defined using the global Sobolev norm 
and the other one using the Sobolev norm on $\Om$.
When $\Om$ is open, these two capacities are easily shown 
to have the same zero sets.
We shall show that the same holds also if $\Om$ is only quasiopen,
i.e.\ open up to sets of arbitrarily small capacity,
see Proposition~\ref{prop-CpU} for the exact details.

To consider Sobolev spaces and capacities on nonopen
sets 
is natural e.g.\ in fine potential theory.
Such studies were pursued on quasiopen sets in $\R^n$ by
Kilpel\"ainen--Mal\'y~\cite{KiMa92},
Latvala~\cite{LatPhD} and Mal\'y--Ziemer~\cite{MZ}.
In the last two decades, 
several types of Sobolev spaces have been introduced on 
general metric spaces by e.g.\ Cheeger~\cite{Cheeger}, Haj\l asz~\cite{Haj03} 
and Shanmugalingam~\cite{Sh-rev}.
Using this 
approach one just regards 
subsets as metric spaces in their
own right, with the metric and the measure 
inherited from the underlying space.
This makes it possible to define Sobolev type spaces on 
even more general subsets.

We shall use the Newtonian Sobolev spaces, which 
on open subsets of $\R^n$  are known
to coincide with the usual Sobolev spaces,
see Theorem~4.5 in Shanmugalingam~\cite{Sh-rev} and 
Theorem~7.13 in Haj\l asz~\cite{Haj03}. 
This equivalence is true also for open subsets 
of weighted $\R^n$ with \p-admissible weights, $p>1$, see
Propositions~A.12 and~A.13 in \cite{BBbook}.
See also Heinonen--Koskela--Shammugalingam--Tyson~\cite{HKST} for more
on Newtonian spaces.

It is well known that every equivalence class of the classical Sobolev spaces
contains better-than-usual, so-called \emph{quasicontinuous}, representatives.
In metric spaces this is only known to hold under certain
assumptions.
Roughly speaking, quasicontinuity means continuity outside sets
of arbitrarily small capacity, see Definition~\ref{def-qcont}.
The Sobolev capacity plays a central role when defining 
quasicontinuity, and there are
actually \emph{two types} of quasicontinuity that one can consider on $\Om$,
one for each of the two capacities mentioned above.
As far as we know this subtle distinction has not been discussed
in the literature.
It is not difficult to show that these two notions of quasicontinuity
are equivalent if $\Om$ is open, and examples show that
for general sets this is not true.

We shall show in Proposition~\ref{prop-qcont-char}
that the equivalence holds for functions defined on
quasiopen sets.
The proof is  
more involved than for open sets, but 
still rather elementary, and it holds in 
arbitrary metric spaces (only assuming 
that balls  have finite measure).
In Proposition~\ref{prop-quasiopen-char}
we obtain
a similar equivalence
for two notions of quasiopenness.
These two results (and also 
Proposition~\ref{prop-CpU}) complement
the restriction result from  
Bj\"orn--Bj\"orn~\cite[Proposition~3.5]{BBnonopen}, stating that if 
$U\subset X$ is \p-path open and measurable, 
then the minimal \p-weak upper gradients 
with respect to $X$ and $U$ coincide   in~$U$.
All these results show the equivalence between a global property
and the corresponding property 
localized to a quasiopen or
\p-path open set.

It was shown
by Shanmugalingam~\cite[Remark~3.5]{Sh-harm} that quasiopen sets in arbitrary
metric spaces are \p-\emph{path open},
i.e.\ that \p-almost every rectifiable curve meets such a set in 
a relatively open 1-dimensional set.
We shall show that under the usual assumptions on the metric space
(and in particular in $\R^n$),
the converse implication is true as well.
More precisely, we prove the following result.

\begin{thm}  \label{thm-p-path-quasiopen}
Assume that the metric space $X$ equipped with a doubling measure 
$\mu$ is complete and supports  a \p-Poincar\'e inequality.
Then every \p-path open set in $X$ is quasiopen\/
\textup{(}and in particular measurable\/\textup{)}.
\end{thm}

Under the same assumptions it was recently
shown by Bj\"orn--Bj\"orn--Latvala~\cite[Theorem~1.4]{BBLat2}
that a set is quasiopen if and only if it is a union of a finely open
set and a set with zero capacity, generalizing a similar result from $\R^n$,
see Adams--Lewis~\cite[Proposition~3]{AdLew}.
Thus we now have two characterizations of quasiopen sets.

As a consequence of Theorem~\ref{thm-p-path-quasiopen}
we obtain the following characterization of quasicontinuous functions.

\begin{thm} \label{thm-qcont=>composition-intro} 
Assume that the metric space $X$ equipped with a doubling measure 
$\mu$ is complete and supports  a \p-Poincar\'e inequality.
Let $U \subset X$ be quasiopen.

Then $u:U \to [-\infty,\infty]$
is quasicontinuous 
if and only if
it is measurable and finite q.e., and 
$u\circ\ga$ is continuous for \p-a.e.\ curve $\ga:[0,l_\ga]\to U$.
\end{thm}

In Proposition~\ref{prop-qcont-char}, we also provide a characterization of 
quasicontinuity using quasiopen sets in the spirit of 
Fuglede~\cite[Lemma~3.3]{Fugl71}.
In \cite[Theorem~1.4]{BBLat2} yet another characterization of quasicontinuity
was given, this time in terms of fine continuity,
see also 
\cite[Lemma, p.\ 143]{Fugl71}.

Newtonian functions are defined more precisely
than the usual Sobolev functions 
(in the sense that the classes of representatives 
are narrower),
and under
the assumptions of Theorem~\ref{thm-p-path-quasiopen}, it was shown in
Bj\"orn--Bj\"orn--Shan\-mu\-ga\-lin\-gam~\cite{BBS5} 
that \emph{all} Newtonian functions on $X$ and on open 
subsets of $X$ are quasicontinuous.
Moreover, the recent results in Ambrosio--Colombo--Di Marino~\cite{AmbCD}
and Ambrosio--Gigli--Savar\'e~\cite{AmbGS} imply that the same holds if
$X$ is  a complete doubling metric space and $1<p<\infty$.

Using the characterization in
Theorem~\ref{thm-qcont=>composition-intro} we can extend the quasicontinuity
result from \cite{BBS5}
to quasiopen sets as follows.
See also Bj\"orn--Bj\"orn--Latvala~\cite{BBLat3} and 
Remark~\ref{rmk-quasi-Lindelof}.

\begin{thm} \label{thm-qcont-U}
Assume that the metric space $X$ equipped with a doubling measure 
$\mu$ is complete and supports  a \p-Poincar\'e inequality.
Let $U \subset X$ be quasiopen.
Then every function $u \in \Nploc(U)$ is quasicontinuous.
\end{thm}

In the last section we weaken the 
assumptions in Theorem~\ref{thm-p-path-quasiopen}
and replace the doubling property and the Poincar\'e inequality
by the requirement that bounded Newtonian functions
are quasicontinuous, which is a much weaker 
assumption. 
In particular, we obtain the following result.
In Section~\ref{sect-no-PI} we give a more general 
version using coanalytic sets.

\begin{thm} \label{thm-p-path-quasiopen-no-PI-intro}
Assume that $X$ is complete, and that every 
bounded $u \in \Np(X)$ is quasicontinuous.
Then every Borel
\p-path open set $U\subset X$ is quasiopen.
\end{thm}

We also prove
a similar modification of Theorem~\ref{thm-qcont=>composition-intro},
see Proposition~\ref{prop-qcont=>composition-no-PI}.
Our proofs of these generalized results (without 
doubling and Poincar\'e assumptions)
are based on the following result
which guarantees measurability
of certain functions defined by their upper gradients.
It may be of independent interest, and generalizes Corollary~1.10 from
J\"arvenp\"a\"a--J\"arvenp\"a\"a--Rogovin--Rogovin--Shanmugalingam~\cite{JJRRS},
where a similar measurability result was proved in the singleton
case $X\setm U =\{x_0\}$.

\begin{prop} \label{prop-meas-for-anal-Borel}
Assume that $X$ is complete  and separable, and that
$U$ is a co\-ana\-lytic set.
For every Borel function $\rho:X\to[0,\infty]$ define
\[
u_{\rho}(x) =  \inf_{\ga\in\Gamma_x} \int_\ga \rho\,ds,
\]
where $\Gamma_x$ is the family of  
all rectifiable curves $\ga: [0,l_\ga] \to X$ 
\textup{(}including constant
curves\/\textup{)} such that $\ga(0)=x$ and $\ga(l_\ga)\in X\setm U$.
Then $u_{\rho}$ is measurable.
\end{prop}

The paper is organized as follows.
In Section~\ref{sect-prelim} we recall the necessary background on Newtonian
spaces.
Section~\ref{sect-qcont} deals with the two notions of quasicontinuity.
The results in that section are valid in arbitrary metric spaces.
Theorems~\ref{thm-p-path-quasiopen}--\ref{thm-qcont-U} are
proved in Section~\ref{sect-p-path-open}. 
In Section~\ref{sect-no-PI} some partial generalizations of the results 
from Section~\ref{sect-p-path-open} are proved without the 
Poincar\'e and doubling
assumptions.
We also formulate two open problems about Borel representatives of
Newtonian functions.

\begin{ack}
The  first two authors were supported by the Swedish Research Council.
Part of this research was done while J.~B. visited
the Charles University in Prague in 2014; 
she thanks the Department of Mathematical Analysis for support
and hospitality. 
\end{ack}

\section{Notation and preliminaries}
\label{sect-prelim}

We assume throughout this paper that $1 \le p<\infty$
and that $X=(X,d,\mu)$ is a metric space equipped
with a metric $d$ and a positive complete  Borel  measure $\mu$
such that $\mu(B)<\infty$ for all open balls $B \subset X$.

A \emph{curve} is a continuous mapping from an interval,
and a \emph{rectifiable} curve $\ga$ is a curve with finite length $l_\ga$.
We will only consider curves which are compact and rectifiable.
Unless otherwise stated they will also be nonconstant
and parameterized by arc length $ds$.
We follow Heinonen and Koskela~\cite{HeKo98} in introducing
upper gradients as follows (they called them very weak gradients).

\begin{deff} \label{deff-ug}
A nonnegative Borel function $g$ on $X$ is an \emph{upper gradient}
of an extended real-valued function $f$
on $X$ if for all nonconstant, compact and rectifiable curves
$\gamma: [0,l_{\gamma}] \to X$,
\begin{equation} \label{ug-cond}
        |f(\gamma(0)) - f(\gamma(l_{\gamma}))| \le \int_{\gamma} g\,ds,
\end{equation}
where we follow the convention that the left-hand side is $\infty$
whenever at least one of the 
terms therein is infinite.
If $g$ is a nonnegative measurable function on $X$
and if \eqref{ug-cond} holds for \p-almost every curve (see below),
then $g$ is a \emph{\p-weak upper gradient} of~$f$.
\end{deff}

Here we say that a property holds for \emph{\p-almost every curve}
if it fails only for a curve family $\Ga$ with zero \p-modulus,
i.e.\ there exists $0\le\rho\in L^p(X)$ such that
$\int_\ga \rho\,ds=\infty$ for every curve $\ga\in\Ga$.
Note that a \p-weak upper gradient \emph{need not} be a Borel function,
it is only required to be measurable.

The \p-weak upper gradients were introduced in
Koskela--MacManus~\cite{KoMc}. It was also shown there
that if $g \in \Lp(X)$ is a \p-weak upper gradient of $f$,
then one can find a sequence $\{g_j\}_{j=1}^\infty$
of upper gradients of $f$ such that $g_j \to g $ in $L^p(X)$.
If $f$ has an upper gradient in $\Lp(X)$, then
it has a \emph{minimal \p-weak upper gradient} $g_f \in \Lp(X)$
in the sense that for every \p-weak upper gradient $g \in \Lp(X)$ of $f$ we have
$g_f \le g$ a.e., see Shan\-mu\-ga\-lin\-gam~\cite{Sh-harm} 
and Haj\l asz~\cite{Haj03}.
The minimal \p-weak upper gradient is well defined
up to a set of measure zero in the cone of nonnegative functions in $\Lp(X)$.
Following Shanmugalingam~\cite{Sh-rev},
we define a version of Sobolev spaces on the metric measure space $X$.

\begin{deff} \label{deff-Np}
Let for measurable $f$,
\[
        \|f\|_{\Np(X)} = \biggl( \int_X |f|^p \, \dmu
                + \inf_g  \int_X g^p \, \dmu \biggr)^{1/p},
\]
where the infimum is taken over all upper gradients $g$ of $f$.
The \emph{Newtonian space} on $X$ is
\[
        \Np (X) = \{f: \|f\|_{\Np(X)} <\infty \}.
\]
\end{deff}

The space $\Np(X)/{\sim}$, where  $f \sim h$ if and only if $\|f-h\|_{\Np(X)}=0$,
is a Banach space and a lattice, see Shan\-mu\-ga\-lin\-gam~\cite{Sh-rev}.
In this paper we assume that functions in $\Np(X)$ are defined everywhere,
not just up to an equivalence class in the corresponding function space.
Nevertheless, we will still say that
$\ut$ is a \emph{representative} of $u$ if $\ut \sim u$.

\begin{deff} \label{deff-sobcap}
The \emph{Sobolev capacity} of an arbitrary set $E\subset X$ is
\[
\Cp(E) = \inf_u\|u\|_{\Np(X)}^p,
\]
where the infimum is taken over all $u \in \Np(X)$ such that
$u\geq 1$ on $E$.
\end{deff}

The capacity is countably subadditive.
A property holds \emph{quasieverywhere} (q.e.)\
if the set of points  for which the property does not hold
has capacity zero.
The capacity is the correct gauge
for distinguishing between two Newtonian functions.
If $u \in \Np(X)$ and $v$ is everywhere defined, 
then $v\sim u$ if and only if $v=u$ q.e.
Moreover, Corollary~3.3 in Shan\-mu\-ga\-lin\-gam~\cite{Sh-rev} shows that
if $u,v \in \Np(X)$ and $u= v$ a.e., then $u=v$ q.e.
In particular, $\ut$ is a representative of $u$ if and only if $\ut=u$ q.e.

We thus see that the equivalence classes in 
$\Np(X)/{\sim}$ are more narrowly defined than for the usual Sobolev spaces.
In weighted  $\R^n$ (with a \p-admissible weight and $p>1$), $\Np(\R^n)/{\sim}$ 
coincides with the refined Sobolev space
as defined in Heinonen--Kilpel\"ainen--Martio~\cite[p.~96]{HeKiMa},
see Bj\"orn--Bj\"orn~\cite[Appendix~A.2]{BBbook}.

For a measurable set $U\subset X$, the Newtonian space $\Np(U)$ is defined by
considering $(U,d|_U,\mu|_U)$ as a metric space in its own right.
It comes naturally with the intrinsic Sobolev capacity that
we denote by $\CpU$.

The measure  $\mu$  is \emph{doubling} if
there exists a \emph{doubling constant} $C>0$ such that for all balls
$B=B(x_0,r):=\{x\in X: d(x,x_0)<r\}$ in~$X$,
\begin{equation*}
        0 < \mu(2B) \le C \mu(B) < \infty,
\end{equation*}
where $\delta B=B(x_0,\delta r)$. 
A metric space with a doubling measure is proper\/
\textup{(}i.e.\ such that
closed and bounded subsets are compact\/\textup{)}
if and only if it is complete.
See Heinonen~\cite{heinonen} for more on doubling measures.

We will also need the following definition.

\begin{deff} \label{def.PI.}
The space $X$ supports a \emph{\p-Poincar\'e inequality} if
there exist constants $C>0$ and $\lambda \ge 1$
such that for all balls $B \subset X$,
all integrable functions $f$ on $X$ and all 
(\p-weak) upper gradients $g$ of $f$,
\[ 
        \vint_{B} |f-f_B| \,\dmu
        \le C \diam(B) \biggl( \vint_{\lambda B} g^{p} \,\dmu \biggr)^{1/p},
\] 
where $ f_B
 :=\vint_B f \,\dmu
:= \int_B f\, d\mu/\mu(B)$.
\end{deff}

See Bj\"orn--Bj\"orn~\cite{BBbook} or 
Heinonen--Koskela--Shanmugalingam--Tyson~\cite{HKST}
for further discussion.

\section{Quasicontinuity and quasiopen sets}
\label{sect-qcont}

We are now ready to define the two 
notions of quasicontinuity considered in this paper.
We let $\Cp$ denote the Sobolev capacity taken with respect 
to the underlying space $X$ and $\CpU$ will be the 
intrinsic Sobolev capacity 
taken with respect to $U$, i.e.\ with $X$ in Definition~\ref{deff-sobcap}
replaced by $U$.

\begin{deff}   \label{def-qcont}
Let $U \subset X$ be measurable.
A function $u: U \to \eR:=[-\infty,\infty]$ 
is \emph{$\CpU$-quasicontinuous \textup{(resp.\ }$\Cp$-quasicontinuous\/\textup{)}}
if for every $\eps>0$ there is a relatively open set $G\subset U$ such that 
$\CpU(G)<\eps$ 
(resp.\ an open set $G \subset X$ such that $\Cp(G)<\eps$)
and such that $u|_{U \setm G}$ is finite and continuous.
\end{deff}

This distinction  was tacitly suppressed in
Bj\"orn--Bj\"orn--Shanmugalingam~\cite{BBS5},
Bj\"orn--Bj\"orn~\cite{BBbook} 
and Bj\"orn--Bj\"orn--Latvala~\cite{BBLat1}--\cite{BBLat3}.
The first two 
deal only with quasicontinuous functions
on open sets, and in this case 
the two definitions are relatively easily shown to be equivalent. 
In this note we show that the same equivalence 
holds for quasiopen $U$
(as considered in \cite{BBLat1}--\cite{BBLat3}); 
the proof of this is 
more involved, although
still rather elementary. 
The equivalence holds  without
any assumptions on the metric space
other than that the measure of balls should be finite.

\begin{deff}   \label{def-q-open}
A set $U\subset X$ is \emph{quasiopen} if for every
$\varepsilon>0$ there is an open set $G\subset X$ such that $\Cp(G)<\varepsilon$
and $G\cup U$ is open.
\end{deff}

Quasiopen sets are measurable by Lemma~9.3 in 
Bj\"orn--Bj\"orn~\cite{BBnonopen}.
It is also quite easy to see that every ($\CpU$ or $\Cp$)-quasicontinuous
function on a measurable set
(and thus in particular on a quasiopen set) is measurable.

The quasiopen sets do not (in general) form a topology.
This is easily seen in unweighted $\R^n$ with $p \le n$ as all singleton
sets are quasiopen, but not all sets are quasiopen.
We shall prove
the following characterizations of quasiopen sets.

\begin{prop} \label{prop-quasiopen-char}
Let $U\subset X$ be a  quasiopen set and 
$V\subset U$.
Then the following statements are equivalent\/\textup{:}
\begin{enumerate}
\item \label{it-q-open-Cp}
$V$ is quasiopen\/ \textup{(}in $X$\textup{)}\/\textup{;}
\item \label{it-q-open-Cp-U}
$V$ is $\Cp$-quasiopen in $U$, i.e.\
for every $\eps>0$ there is a relatively open set $G \subset U$
such that $\Cp(G)<\eps$ and $G \cup V$ is relatively open in $U$\/\textup{;}
\item \label{it-q-open-CpU}
$V$ is $\CpU$-quasiopen in $U$, i.e.\
for every $\eps>0$ there is a relatively open set $G \subset U$
such that $\CpU(G)<\eps$ and $G \cup V$ is relatively open in $U$\/\textup{.}
\end{enumerate}
\end{prop}

Even though the quasiopen sets do not (in general) form a topology,
we still have the
 characterization \ref{it-char} below  of quasicontinuity
using quasiopen sets.

\begin{prop} \label{prop-qcont-char} 
Let $u:U \to \eR$ be a function on a quasiopen set $U\subset X$.
Then the following statements are equivalent\/\textup{:}
\begin{enumerate}
\renewcommand{\theenumi}{\textup{(\roman{enumi})}}%
\item \label{it-qcont}
$u$ is $\CpU$-quasicontinuous\/\textup{;}
\item \label{it-str-qcont}
$u$ is $\Cp$-quasicontinuous.
\end{enumerate}
Moreover, if $X$ is locally compact then 
these statements are equivalent to the following statement\/\textup{:}
\begin{enumerate}
\renewcommand{\theenumi}{\textup{(\roman{enumi})}}%
\addtocounter{enumi}{2}
\item \label{it-char}
$u$ is finite q.e.\ and the sets
$U_\alp:=\{x\in U:u(x)>\al\}$ and 
$V_\alp:=\{x\in U:u(x)<\al\}$, $\alp \in \R$, are quasiopen
{\rm(}in any of the equivalent senses \ref{it-q-open-Cp}, \ref{it-q-open-Cp-U}
and~\ref{it-q-open-CpU} of Proposition~\ref{prop-quasiopen-char}\/{\rm)}.
\end{enumerate}
\end{prop}

\begin{remark} \label{rmk-qcont-char}
The local compactness assumption is only needed
when showing that there is an open neighbourhood
of $\{x:|u(x)|=\infty\}$ with small capacity when
proving \ref{it-char} \imp \ref{it-str-qcont}.
In particular 
\ref{it-qcont} \eqv \ref{it-str-qcont} \eqv \ref{it-char}
holds without this assumption
for real-valued functions $u:U \to \R$.
\end{remark}

\begin{example}
The equivalence between the two types of
quasicontinuity does not hold for arbitrary measurable subsets:
To see this let e.g.\ $U=(\R \setm \Q) ^2$ as a subset of (unweighted)
$\R^2$ and $u=\chi_{[0,\sqrt{2}]^2}$.
Then any relatively open set $G \subset U$ such that $u|_{U \setm G}$ is continuous
must contain at least one point on each
line $\{(x,y):x=t\}$ for $0 < t < \sqrt{2}$, $t \notin \Q$.
It follows, by projection,  that 
$\Cp(G)\ge \Cp(([0,\sqrt{2}]\setm\Q) \times \{0\})>0$
and thus $u$ is not $\Cp$-quasicontinuous.
On the other hand,
let 
$E=U \cap (([0,\sqrt{2}]\times\{\sqrt{2}\}) \cup (\{\sqrt{2}\} \times [0,\sqrt{2}]))$.
Then $u|_{U \setm E}$ is continuous.
Since $\mu(E)=0$,
the regularity of the measure shows that, for every $\eps>0$,
there is a relatively open subset $G$ of $U$ such that $E \subset G$
and $\mu(G)< \eps$.
As there are no nonconstant curves in $U$, we have $\CpU(G)=\mu(G)$,
and thus $u$ is $\CpU$-quasicontinuous.
\end{example}

In Fuglede~\cite[Lemma~3.3]{Fugl71}, a characterization similar to \ref{it-char}
was obtained for general capacities on topological spaces.
The definitions of quasiopen sets and quasicontinuity therein differ however
somewhat from ours.
More precisely, in~\cite{Fugl71}, a set $V$ is called quasiopen if for
every $\eps>0$ there exists an open set $\Om$ such that the symmetric
difference $(V\setm\Om)\cup(\Om\setm V)$ has capacity less than $\eps$.
If the capacity is outer then this notion is easily shown to be equivalent 
to our definition, but in general Fuglede's definition allows for more
quasiopen sets.

Similarly, in the definition of quasicontinuity in~\cite{Fugl71}, it is
not in general required that the removed exceptional set $G$ be open 
(though for outer capacities this can always be arranged) and continuity 
does not require finiteness.
In other words, Fuglede's definition of quasicontinuity corresponds to the 
\emph{weak quasicontinuity} considered (on open sets) in 
Bj\"orn--Bj\"orn~\cite[Section~5.2]{BBbook}, less the requirement that 
continuous functions be finite.
Fuglede's notion of quasicontinuity is in~\cite[Lemma~3.3]{Fugl71} proved
to be equivalent to the fact that the sets $U_\alp$ and 
$V_\alp$, $\alp \in \R$, in Proposition~\ref{prop-qcont-char} are quasiopen 
(in the sense of~\cite{Fugl71}).

Since our definitions do not exactly agree with those in~\cite{Fugl71},
and moreover we consider quasicontinuity and quasiopenness with respect to 
two different capacities ($\Cp$ and $\CpU$) and two possible underlying
spaces ($X$ and $U$), we present here for the reader's convenience full
proofs of these results.

The proof of Proposition~\ref{prop-qcont-char}
can be easily modified to show that on quasiopen sets
weak quasicontinuity is the same with respect to $\Cp$ and $\CpU$.

\begin{proof}[Proof of Proposition~\ref{prop-quasiopen-char}.]
\ref{it-q-open-Cp} \imp \ref{it-q-open-Cp-U}
For every $\eps>0$ there exists an open $G\subset X$ such that $\Cp(G)<\eps$
and $\Om:=V\cup G$ is open. 
Then $\Gt:=G\cap U$ and $\Omt:=\Om\cap U$ are relatively open in $U$ and
$\Cp(\Gt)<\eps$, i.e.\ \ref{it-q-open-Cp-U} holds.

\ref{it-q-open-Cp-U} \imp \ref{it-q-open-CpU}
This is straightforward, since $\CpU$ is majorized by $\Cp$.

\ref{it-q-open-CpU} \imp \ref{it-q-open-Cp}
We may assume that $\emptyset \ne U \ne X$.
Let $\eps >0$ be arbitrary. 
Since $U$ is quasiopen, there is an open set $G$ such 
that $\Cp(G)<\eps$ and such that $\Om:=U \cup G \ne X$ is open.
Hence there is $w \in \Np(X)$ such that $\chi_G \le w \le 1$
and $\|w\|_{\Np(X)}^p < \eps$.

Using  that
$V$ is $\CpU$-quasiopen,
we can find, for each $j=1,2,\ldots$,
a relatively open set $G_j \subset U$ such that
$V\cup G_j$ is relatively open in $U$ and
$\CpU(G_j)<\eps_j:=2^{-j}j^{-p}\eps$.
There is thus $v_j \in \Np(U)$ such that $\chi_{G_j}  \le v_j \le 1$ in $U$
and $\|v_j\|_{\Np(U)}^p < \eps_j$.
Next, let $x_0 \in \Om$,
\[ 
    \Om_j=\{x \in \Om : \dist(x,X \setm \Om)> 1/j
    \text{ and } d(x,x_0) < j\}, 
    \quad j=1,2,\ldots,
\]
and $\eta_j(x)=(1-j\dist(x,\Om_j))_\limplus$.
(If $\Om_j=\emptyset$ we let $\eta_j \equiv 0$.)
Define the functions 
\[
    \phi_k = \begin{cases}
     \displaystyle
     \min \Bigl\{ 1-w, \max_{1 \le j \le k} v_j \eta_j \Bigr\} & \text{in } U, \\
     0, & \text{otherwise},
     \end{cases}
      \quad \text{and} \quad
      \phi=\lim_{k \to \infty} \phi_k.
\]
Then $\phi_k$ has bounded support and 
$\phi_k \in \Np(U)$.
Since 
\[
0 \le \phi_k \le\min\{1- w,\eta_k\} \in \Np_0(U)
:=\{f: f \in \Np(X) \text{ and } f=0 \text{ outside } U\},
\]
Lemma~2.37 in \cite{BBbook} 
implies that $\phi_k \in \Np_0(U) \subset \Np(X)$.
Next,
\begin{equation} \label{eq-phik}
     \int_X \phi_k^p \, d\mu 
     \le \sum_{j=1}^{k}      \int_X v_j^p \, d\mu 
     < \sum_{j=1}^{k}      \eps_j 
     < \eps.
\end{equation}
By Lemma~1.52 in \cite{BBbook}, 
$g:=\sup_k g_{\phi_k}$
is a \p-weak upper gradient of $\phi$.
For a.e.\ $x \in U$ we either
have $g=g_{\phi_k}=g_w$ or $g=g_{\phi_k}=g_{v_j \eta_j} \le g_{v_j} + jv_j$ for
some $j$ and $k$
(see \cite[Theorem~2.15 and Corollary~2.21]{BBbook}).
Hence
\[
    \int_X g^p \, d\mu 
     \le     \int_X g_{w}^p \, d\mu 
            + \sum_{j=1}^{\infty} 2^{p-1} \int_U (g_{v_j}^p + j^p v_j^p)  \, d\mu 
      \le \eps + \sum_{j=1}^{\infty} 2^{p-1} j^p \eps_j 
     < \eps + 2^p\eps,
\]
which together with \eqref{eq-phik} shows that
$\|\phi\|_{\Np(X)}^p \le 2\eps+2^p \eps$.

Next, set
\[
     H= G \cup \bigcup_{j=1}^\infty (G_j \cap \Om_j),
\]
which is open by the choice of $G$ as $\bigcup_{j=1}^\infty (G_j \cap \Om_j)$
is relatively open in $U$.
Then $w+\phi\ge\chi_H$ and hence
\[
   \Cp(H) \le \| w+\phi \|_{\Np(X)}^p 
     \le 2^{p-1} (\|w\|_{\Np(X)}^p + \|\phi\|_{\Np(X)}^p)
     \le 2^{p-1}(3\eps + 2^p\eps).
\]
It remains to show that $V\cup H$ is open in $X$.
If $x\in V$, then $x\in\Om_k$ for some $k$, and as $V\cup G_k$ is 
relatively open in $U$ (and thus $G\cup V\cup G_k$ is open in $X$), 
we can find a ball $B_x$ such that 
\[
x\in B_x\subset (G\cup V\cup G_k) \cap \Om_k \subset V\cup G\cup(G_k\cap\Om_k)
\subset V\cup H.
\]
For $x\in H$, we can instead choose $x\in B_x\subset H \subset V\cup H$.
\end{proof}

\begin{proof}[Proof of Proposition~\ref{prop-qcont-char}.]
\ref{it-str-qcont} \imp \ref{it-qcont}
This is straightforward, since $\CpU$ is majorized by $\Cp$.

\ref{it-qcont} \imp \ref{it-str-qcont}
We may assume that $\emptyset \ne U \ne X$.
Let $\eps >0$ be arbitrary. 
For each $j=1,2,\ldots$, 
there is a relatively open set $G_j \subset U$ such that
$\CpU(G_j)<\eps_j:=2^{-j}j^{-p}\eps$ 
and such that $u|_{U \setm G_j}$ is finite and continuous in $U$.

Next construct the open sets $G$, $H$ and $\Om_k$, $k=1,2,\ldots$,  as in the proof of
Proposition~\ref{prop-quasiopen-char}, \ref{it-q-open-CpU} \imp \ref{it-q-open-Cp}.
Then $\Cp(H) \le 2^{p-1}(3\eps + 2^p\eps)$. 
If
$x \in U \setm H$,
then there is $k$ such that $x \in \Om_k$.
Since $u|_{U \setm G_k}$ is finite and continuous at $x$, it
thus follows that
$u|_{U \setm H}$ is finite and continuous at $x$.
Hence $u|_{U \setm H}$ is finite and continuous,
and $u$ is $\Cp$-quasicontinuous.

\medskip
Now assume that $X$ is locally compact.

\ref{it-str-qcont} \imp \ref{it-char}
Let $\eps>0$.
Since $U$ is quasiopen we can find an open set $G$ such that
$G \cup U$ is open and $\Cp(G) <\eps$.
As $u$ is $\Cp$-quasicontinuous,
 there is an open set $H \subset X$ such that $\Cp(H) <\eps$
and $u|_{U \setm H}$ is continuous.
Thus, for every $\alp \in \R$,
$U_\alp \setm H$ is relatively open in $U \setm H$.
Hence $U_\alp \cup (G \cup H)$ is open and $\Cp(G \cup H)< 2\eps$,
showing that $U_\alp$ is quasiopen.
That $V_\alp$ is quasiopen follows similarly, whereas
$u$ is finite q.e.\ by definition.

\ref{it-char} \imp \ref{it-str-qcont}
Let $\eps>0$ and $E=\{x \in U : |u(x)|=\infty\}$.
By assumption, $\Cp(E)=0$ and thus
there is an open set $H \supset E$ such that $\Cp(H)<\eps$,
by Proposition~1.4 in Bj\"orn--Bj\"orn--Shanmugalingam~\cite{BBS5} and 
Proposition~4.7  in Bj\"orn--Bj\"orn--Lehrb\"ack~\cite{BBLeh1}.
Next let $\{q_j\}_{j=1}^\infty$ be an enumeration of $\Q$.
By assumption there are open $G_j$ such that $U_{q_j} \cup G_j$ 
and $V_{q_j} \cup G_j$ are open
and $\Cp(G_j) < 2^{-j} \eps$.
Let $G=H \cup \bigcup_{j=1}^\infty G_j$, which is open and such that
$\Cp(G) < 2\eps$.
Moreover, $U_{q_j} \setm G$ is relatively open in $U \setm G$.
For $\alp \in \R$ it follows that
\[
  U_\alp \setm G = \bigcup_{\Q\ni q > \alp}
(U_q \setm G)
\]
is relatively open in $U \setm G$.
Similarly $V_\alp \setm G$ is relatively open in $U \setm G$, 
and thus $u|_{U \setm G}$ is finite and continuous.
\end{proof}

Theorem~\ref{thm-p-path-quasiopen}
shows that under certain assumptions on $X$, \p-path open
sets are quasiopen, and thus 
Propositions~\ref{prop-quasiopen-char} 
and~\ref{prop-qcont-char}
hold for \p-path open sets in that case. 
In general, \p-path open sets need not be measurable, see
Example~\ref{ex-qcont-not-p-path=qopen} below.
It would therefore be interesting to know if 
the conclusions of Propositions~\ref{prop-quasiopen-char} 
and~\ref{prop-qcont-char}
hold for measurable \p-path open sets $U$ 
(or even measurable \p-path almost open sets $U$, see 
Bj\"orn--Bj\"orn~\cite{BBnonopen})
without additional assumptions on $X$.
In Section~\ref{sect-no-PI} some partial results are obtained
for \p-path open sets which are Borel.
Note that in the situation described in 
Example~\ref{ex-qcont-not-p-path=qopen} below, the conclusions of
Propositions~\ref{prop-quasiopen-char} and~\ref{prop-qcont-char}
do hold for measurable \p-path open sets.

\section{\texorpdfstring{\p-path}{path} open and quasiopen sets}
\label{sect-p-path-open}

\begin{deff}
\label{def-p-path-open}
A set $G\subset X$ is \p-\emph{path open}
(in $X$) if for \p-almost every curve
$\ga:[0,l_\ga]\to X$, the set $\ga^{-1}(G)$ 
is\/ \textup{(}relatively\/\textup{)} open in $[0,l_\ga]$.
\end{deff}

The arguments proving
Propositions~\ref{prop-quasiopen-char} and~\ref{prop-qcont-char} can also be used to 
show that $\Cp$ and $\CpU$ have the same zero sets for quasiopen $U$.
Using a different approach we can 
obtain this
more generally for \p-path open sets.

\begin{prop} \label{prop-CpU}
Let $U$ be a measurable \p-path open set and $E \subset U$.
Then
 $\Cp(E)=0$ if and only if $\CpU(E)=0$.
\end{prop}

For open $U$ this is well-known, see e.g.\  Lemma~2.24 in 
Bj\"orn--Bj\"orn~\cite{BBbook}.

\begin{proof}
Assume that $\CpU(E)=0$.
Proposition~1.48 in \cite{BBbook}, applied with $U$ as
the underlying space, implies that $\mu(E)=0$ and that \p-almost every
curve in $U$ avoids $E$.
Since $U$ is \p-path open, it follows that also \p-almost every curve in $X$
avoids $E$ (see Lemma~3.3 in Bj\"orn--Bj\"orn--Latvala~\cite{BBLat3}).
Thus, by Proposition~1.48 in \cite{BBbook} again (this time
with respect to $X$), $\Cp(E)=0$.

The converse implication is trivial.
\end{proof}

The \p-path open sets were introduced by 
Shanmugalingam~\cite[Remark~3.5]{Sh-harm}. 
It was also shown there that every quasiopen set is \p-path open.
We are now going to prove Theorem~\ref{thm-p-path-quasiopen}
which says that the converse is true
under suitable assumptions on $X$.
In particular, this
holds in $\R^n$.

\begin{proof}[Proof of Theorem~\ref{thm-p-path-quasiopen}]
Let $U\subset X$ be \p-path open.
Then the family $\Gamma$ of curves $\ga$ in $X$, for which $\ga^{-1}(U)$
is not relatively open, has zero \p-modulus, i.e.\ there exists 
$\rho\in L^p(X)$ such that $\int_\ga \rho\,ds=\infty$ for every $\ga\in\Gamma$.

Assume to start with that $U$ is bounded and let $B$ be a ball containing
$U$.
Define, for $x\in X$,
\begin{equation}  \label{eq-def-u-from-rho}
u(x) = \min \biggl\{1, \inf_\ga \int_\ga (\rho+\chi_B)\,ds \biggr\},
\end{equation}
where $\chi_B$ is the characteristic function of $B$ and
the infimum is taken over all rectifiable curves $\ga: [0,l_\ga] \to X$ 
(including constant
curves) such that $\ga(0)=x$ and $\ga(l_\ga)\in X\setm U$.
Then $u=0$ in $X\setm U$ and
Lemma~3.1 in Bj\"orn--Bj\"orn--Shanmugalingam~\cite{BBS5} (or 
\cite[Lemma~5.25]{BBbook})
shows that $u$ has $\rho+\chi_B$ as an upper gradient.
Since the measure $\mu$ is doubling and $X$ is complete
and supports a \p-Poincar\'e inequality, we can conclude from Theorem~1.11
in J\"arvenp\"a\"a--J\"arvenp\"a\"a--Rogovin--Rogovin--Shanmugalingam~\cite{JJRRS}
that $u$ is measurable.
As $U$ is assumed to be bounded and $\rho\in L^p(X)$, 
it follows that $u\in\Np(X)$.

We claim that $u>0$ in $U$, i.e.\ that $U=\{x\in X:u(x)>0\}$ is a superlevel
set of a Newtonian function.
The assumptions on $X$ guarantee
that $u$ is quasicontinuous (by 
\cite[Theorem~1.1]{BBS5} or \cite[Theorem~5.29]{BBbook}), 
which then directly implies that $U$ is quasiopen, 
see Proposition~\ref{prop-qcont-char}.

To prove the claim, 
let $x\in U$ 
and assume for a contradiction that $u(x)=0$.
Then there exist curves $\ga_j$ connecting $x$ to $X\setm U$ such that
\begin{equation}  \label{eq-choose-ga-j}
\int_{\ga_j} (\rho+\chi_B)\,ds \le 2^{-j}, \quad j=1,2,\ldots.
\end{equation}
In particular, $l_{\ga_j}\le 2^{-j}$ for all $j=1,2,\ldots$\,.

We define a curve $\gat$ as a recursive concatenation of all $\ga_j$ and
their reversed reparameterizations as follows.
Let $L_0=0$,
\[
L_j=2\sum_{i=1}^j l_{\ga_i} \le 2\sum_{i=1}^\infty l_{\ga_i} =: L
\le 2\sum_{j=1}^\infty 2^{-j} =2
\] 
and
\[
\gat(t) = \begin{cases}
          \ga_j (t-L_{j-1}), 
          & \text{if } L_{j-1} \le t \le L_{j-1} + l_{\ga_j}, \ j=1,2,\ldots,\\
          \ga_j (L_j-t), 
          & \text{if } L_{j-1} + l_{\ga_j} \le t \le L_j, \ j=1,2,\ldots.
          \end{cases}
\]
Then $\gat:[0,L]\to X$, $\gat(L)=x$ and 
$\gat(L_{j} + l_{\ga_{j+1}})\in X\setm U$, $j=1,2,\ldots$\,.
Since $x\in U$ and $L_j + l_{\ga_{j+1}} \to L$, as $j\to\infty$, this shows that 
$\gat^{-1}(U)$ is not relatively open in $[0,L]$
and hence $\gat\in\Gamma$.
But \eqref{eq-choose-ga-j} implies that
\[
\int_\gat \rho \,ds \le 2 \sum_{j=1}^\infty \int_{\ga_j} (\rho+\chi_B)\,ds 
\le 2\sum_{j=1}^\infty 2^{-j} =2,
\]
contradicting the choice of $\rho$.
We can therefore conclude that $u(x)>0$ for all $x\in U$, which finishes
the proof for bounded $U$.

If $U$ is unbounded, then the above argument shows that $U\cap B_j$ is 
quasiopen, $j=1,2,\ldots$, where $\{B_j\}_{j=1}^\infty$ is a countable cover
of $X$ by (open) balls.
In particular, given $\eps>0$, there exists for every $j=1,2,\ldots$ an open
set  $G_j$ such that $\Cp(G_j)<2^{-j}\eps$ and $(U\cap B_j)\cup G_j$
is open. 
Setting $G=\bigcup_{j=1}^\infty G_j$, we see that $\Cp(G)<\eps$ and
$U\cup G = \bigcup_{j=1}^\infty ((U\cap B_j)\cup G_j)$ is open, 
which concludes the proof.
\end{proof}

We now turn to Theorem~\ref{thm-qcont=>composition-intro} and 
restate it in a more general form.

\begin{prop} \label{prop-qcont=>composition} 
Assume that $X$ is locally compact and that every 
measurable \p-path open set in $X$ is
quasiopen.
Let $u:U \to \eR$ be a function on a quasiopen set~$U$. 
Then the following are equivalent\/\textup{:}
\begin{enumerate}
\item \label{a-1}
$u$ is quasicontinuous\/ \textup{(}with respect to $\Cp$ or $\CpU$\textup{)}\textup{;}
\item \label{a-2}
$u$ is measurable and finite q.e., and 
$u\circ\ga$ is continuous\/ 
\textup{(}on the set where it is defined\/\textup{)}
for \p-a.e.\ curve $\ga:[0,l_\ga]\to X$\textup{;}
\item \label{a-3}
$u$ is measurable and finite q.e., and 
$u\circ\ga$ is continuous for \p-a.e.\ curve $\ga:[0,l_\ga]\to U$.
\end{enumerate}
\end{prop}

The assumptions of Proposition~\ref{prop-qcont=>composition} 
are guaranteed by Theorem~\ref{thm-p-path-quasiopen}, 
but there are also other situations when they hold,
see e.g.\ Example~\ref{ex-qcont-not-p-path=qopen} below and 
Example~5.6 in \cite{BBbook}.

As seen from the proof below and Remark~\ref{rmk-qcont-char}, 
no assumptions on $X$ are needed for the 
implications \ref{a-1} \imp \ref{a-2} \eqv \ref{a-3}
in Proposition~\ref{prop-qcont=>composition}.
Also, if we know that \emph{every} \p-path open set is quasiopen 
(and thus measurable),
then the measurability assumption for $u$ can be dropped,
cf.\ the proof of Theorem~\ref{thm-p-path-quasiopen} where measurability of $u$
follows from Theorem~1.11
in J\"arvenp\"a\"a--J\"arvenp\"a\"a--Rogovin--Rogovin--Shanmugalingam~\cite{JJRRS}. 

For real-valued $u$, the assumption of local compactness  
here and in 
Corollary~\ref{cor-all-Np-qcont} below
can be omitted, see Remark~\ref{rmk-qcont-char}.

\begin{proof}
\ref{a-1} \imp \ref{a-2}
Assume that $u:U \to \eR$ is quasicontinuous. 
Proposition~\ref{prop-qcont-char} shows that it is finite q.e.\ and
the sets $U_\alp:=\{x\in U:u(x)>\al\}$ and 
$V_\alp:=\{x\in U:u(x)<\al\}$, $\alp \in \R$, are quasiopen.
Remark~3.5 in Shanmugalingam~\cite{Sh-harm} implies that $U$,  $U_\al$ and 
$V_\al$, $\al\in\R$, are \p-path open.
Hence, for \p-a.e.\ curve $\ga:[0,l_\ga]\to X$, the set $I_\ga=\ga^{-1}(U)$
is a relatively open subset of $[0,l_\ga]$, and so are the level sets
\begin{align}  \label{eq-level-sets-a.e.-curve}
\ga^{-1}(U_\al) &= \{t\in I_\ga: (u\circ\ga)(t)>\al\}
\quad \text{and} \quad
\ga^{-1}(V_\al) = \{t\in I_\ga: (u\circ\ga)(t)<\al\}, 
\end{align}
for all $\al\in\Q$.
This implies that $u\circ\ga: I_\ga\to\eR$ is continuous for \p-a.e.\ curve $\ga$.

\ref{a-2} \imp \ref{a-3} This is trivial.

\ref{a-3} \imp \ref{a-2}
Let $\Ga$ be the family of exceptional curves $\ga:[0,l_\ga] \to U$ for which 
$u\circ\ga$ is not continuous.
As $U$ is quasiopen, and thus \p-path open  by Remark~3.5 in~\cite{Sh-harm},
$\ga^{-1}(U)$ is open for \p-a.e.\ curve $\ga:[0,l_\ga] \to X$.
By Lemma~1.34 in \cite{BBbook}, \p-a.e.\ such curve $\ga$
does not have a subcurve in $\Ga$.
For such a $\ga$, the composition $u \circ \ga$ is continuous
on the relatively open subset of $[0,l_\ga]$ where it is defined.

\ref{a-2} \imp \ref{a-1}
Since $U$ is quasiopen, and thus \p-path open by Remark~3.5 in~\cite{Sh-harm},
for \p-a.e.\ curve $\ga:[0,l_\ga]\to X$, the set
$I_\ga=\ga^{-1}(U)$ is a relatively open subset 
of $[0,l_\ga]$ and $u\circ\ga$ is continuous on $I_\ga$.
For such $\ga$, and all $\al \in \R$, the level sets 
in~\eqref{eq-level-sets-a.e.-curve} are 
relatively open in $I_\ga$, and thus in $[0,l_\ga]$.
It follows that the level sets $U_\alp:=\{x\in U:u(x)>\al\}$ and 
$V_\alp:=\{x\in U:u(x)<\al\}$, $\alp \in \R$, are \p-path open
and measurable (as $u$ is measurable), and thus
quasiopen by the assumption.
Proposition~\ref{prop-qcont-char} now concludes the proof.
\end{proof}

\begin{proof}[Proof of Theorem~\ref{thm-qcont=>composition-intro}]
Theorem~\ref{thm-p-path-quasiopen} guarantees that the
assumptions of Proposition~\ref{prop-qcont=>composition}
are satisfied, from which the result follows.
\end{proof}

As a consequence of the characterization 
in Proposition~\ref{prop-qcont=>composition}
we obtain the following 
generalization of Theorem~\ref{thm-qcont-U} and 
partial converse
of Theorem~\ref{thm-p-path-quasiopen-no-PI-intro}.

\begin{cor}  \label{cor-all-Np-qcont}
Assume that $X$ is locally compact and that every 
measurable \p-path open set in $X$ is
quasiopen.
Let $U \subset X$ be quasiopen.
Then all $u\in\Nploc(U)$ are quasicontinuous.
\end{cor}

Note that in particular we may let $U=X$.

\begin{proof}
Since every Newtonian function is measurable, finite q.e.\
(by \cite[Proposition~1.30]{BBbook}) and
absolutely continuous on \p-a.e.\ 
curve, by Shanmugalingam~\cite[Proposition~3.1]{Sh-rev} (or 
\cite[Theorem~1.56]{BBbook}), this follows
directly from Proposition~\ref{prop-qcont=>composition}.
\end{proof}

\begin{proof}[Proof of Theorem~\ref{thm-qcont-U}.]
This follows directly from Theorem~\ref{thm-p-path-quasiopen}
and Corollary~\ref{cor-all-Np-qcont}.
\end{proof}

\begin{remark} \label{rmk-quasi-Lindelof}
Theorem~\ref{thm-qcont-U} was recently obtained by 
Bj\"orn--Bj\"orn--Latvala~\cite{BBLat3}. 
The proof given here is very different from that in~\cite{BBLat3}
and appears as an immediate corollary
of other results. In particular, it does not use
the fine topology and the quasi-Lindel\"of principle, whose proof in \cite{BBLat3}
relies on the vector-valued, so-called Cheeger, differentiable structure.

The assumptions of our 
Corollary~\ref{cor-all-Np-qcont} are weaker than those of the version of
Theorem~\ref{thm-qcont-U} in \cite{BBLat3}.
On the other hand, in \cite{BBLat3},
quasicontinuity is deduced for
a larger local space than $\Nploc(U)$.
See \cite{BBLat3} for more details and the precise
definitions of the local spaces considered therein.
\end{remark}

The following example shows that \p-path open sets need not be quasiopen in general.

\begin{example}  \label{ex-qcont-not-p-path=qopen}
Assume that there are no nonconstant rectifiable curves in $X$.
For example, consider $\R$ with the snowflaked metric $d(x,y)=|x-y|^\al$, 
$0<\al<1$, and the Lebesgue measure.

Then every set is \p-path open, while Lemma~9.3 in Bj\"orn--Bj\"orn~\cite{BBnonopen} 
shows that quasiopen sets must be measurable.
Any nonmeasurable set  $U\subset X$ is thus \p-path open but  cannot be quasiopen.
Indeed, the function $u$ constructed in the proof of Theorem~\ref{thm-p-path-quasiopen}
is $\chi_U$, which has zero as an upper gradient, but it is not measurable and
thus not in $\Np(X)$.

Note that in this case the zero function is an upper gradient of every function
and thus $\Np(X)=L^p(X)$ and $\Cp$ is the extension of the measure 
$\mu$ to all subsets of $X$ as an outer measure.
It thus follows that the quasiopen sets are just the measurable sets, and hence
every \emph{measurable} \p-path open set is quasiopen.
Thus Corollary~\ref{cor-all-Np-qcont} applies in this case, 
but this already follows (in this particular case)
from Luzin's theorem.

The same argument applies also if 
the family of nonconstant rectifiable curves has zero \p-modulus.
(In this case ``upper gradient'' should be replaced by ``\p-weak upper gradient'' 
above.)
In fact, this assumption is equivalent to the equality
$\Np(X)=L^p(X)$ as sets of functions, see
L.~Mal\'y~\cite[Lemma~2.5]{lmaly3}.
\end{example}

\section{Outside the realm of a Poincar\'e inequality}
\label{sect-no-PI}

The doubling condition and the Poincar\'e inequality 
are standard assumptions in analysis on metric spaces. 
In particular, they guarantee that Lipschitz functions are dense in $\Np(X)$,
which in turn (together with completeness)
implies the quasicontinuity of Newtonian functions used in the proof
of Theorem~\ref{thm-p-path-quasiopen}.
On the other hand, there are plenty of spaces where the Poincar\'e inequality
fails or the doubling condition is violated, 
but where Newtonian functions are quasicontinuous.
Therefore, in this section we relax the assumptions of 
Theorem~\ref{thm-p-path-quasiopen} and 
obtain Theorem~\ref{thm-p-path-quasiopen-no-PI-intro}.

We postpone the proof of Theorem~\ref{thm-p-path-quasiopen-no-PI-intro} to the
end of this section.
Meanwhile, we formulate some consequences of it.
First, we have the following characterization of quasicontinuity among
Borel functions, analogous to Proposition~\ref{prop-qcont=>composition}. 
Note that 
the assumption of quasiopenness of \p-path open sets
can be deduced
using Theorem~\ref{thm-p-path-quasiopen-no-PI-intro},
or its generalization Theorem~\ref{thm-p-path-quasiopen-no-PI-gen} below,
under the assumptions therein.

\begin{prop} \label{prop-qcont=>composition-no-PI} 
Assume that $X$ is 
locally compact
and that every Borel \p-path open set in $X$ is quasiopen.
Let $u:U \to \eR$ be a Borel function on a quasiopen set~$U$. 
Then the following are equivalent\/\textup{:}
\begin{enumerate}
\item \label{B-1}
$u$ is quasicontinuous\/ \textup{(}with respect to $\Cp$ or $\CpU$\textup{)}\textup{;}
\item \label{B-2}
$u$ is finite q.e.\ and 
$u\circ\ga$ is continuous\/ 
\textup{(}on the set where it is defined\/\textup{)}
for \p-a.e.\ curve $\ga:[0,l_\ga]\to X$\textup{;}
\item \label{B-3}
$u$ is finite q.e.\ and 
$u\circ\ga$ is continuous for \p-a.e.\ curve $\ga:[0,l_\ga]\to U$.
\end{enumerate}
\end{prop}

\begin{proof}
As mentioned before the proof of Proposition~\ref{prop-qcont=>composition},
the proof of \ref{a-1} \imp \ref{a-2} \eqv \ref{a-3} therein
holds without any assumptions on $X$ and therefore
applies also  here.

\ref{B-2} \imp \ref{B-1}
Replace the word ``measurable'' by ``Borel''
twice in the proof of this implication in
the proof of Proposition~\ref{prop-qcont=>composition}.
\end{proof}

\begin{cor}  \label{cor-q-open-Borel-equiv-qcont}
Assume that $X$ is complete and locally compact.
Consider the following statements\/\textup{:}
\begin{enumerate}
\item \label{it-bdd-qcont}
every bounded $u\in\Np(X)$ is quasicontinuous\/\textup{;}
\item \label{it-qcont-Np}
every $u\in\Nploc(X)$ is quasicontinuous\/\textup{;}
\item \label{it-Borel-q-open}
every Borel \p-path open set in $X$ is quasiopen\/\textup{;}
\item \label{it-Borel-qcont}
if $U$ is quasiopen, 
  then every Borel function $u\in\Nploc(U)$ is quasicontinuous.
\end{enumerate}
Then \ref{it-bdd-qcont} $\eqv$  \ref{it-qcont-Np} 
$\imp$ \ref{it-Borel-q-open} $\imp$ \ref{it-Borel-qcont}.

Moreover, if every bounded function $u\in\Np(X)$ has a Borel 
representative $\ut \in \Np(X)$, then all the statements are equivalent.
\end{cor}

Recall that $\ut$ is termed a representative of $u$ if 
both functions
belong to the same equivalence class in $\Np(X)/{\sim}$
(or, equivalently, if $\ut=u$ q.e., see Section~\ref{sect-prelim}).

\begin{proof}
\ref{it-bdd-qcont} \imp \ref{it-qcont-Np}
Let $u \in \Nploc(X)$.
Fix $x_0 \in X$ and 
let $\eta_j \in \Lip(X)$ be such that $\eta_j =1$ on $B(x_0,j)$
and $\eta_j=0$ outside $B(x_0,2j)$.
Then $u_j:=u\eta_j \in \Np(X)$ and also $\arctan u_j \in \Np(X)$.
By assumption $\arctan u_j$ is quasicontinuous.
As $u_j$ is finite q.e., 
Proposition~1.4
in \cite{BBS5} and Proposition~4.7 in \cite{BBLeh1}
show that also $u_j$ is quasicontinuous in $X$.
Hence $u$ is quasicontinuous in $B(x_0,j)$ for each $j$,
and it follows from Lemma~5.18 in \cite{BBbook} that $u$ is
quasicontinuous in $X$.

\ref{it-qcont-Np} \imp \ref{it-bdd-qcont}
This is trivial.

\ref{it-bdd-qcont} \imp \ref{it-Borel-q-open}
This is Theorem~\ref{thm-p-path-quasiopen-no-PI-intro}.

\ref{it-Borel-q-open} \imp \ref{it-Borel-qcont}
Let $u\in\Nploc(U)$ be a Borel function.
Then it is finite q.e.\ and absolutely continuous on \p-a.e.\ 
curve (see Shanmugalingam~\cite[Proposition~3.1]{Sh-rev} or 
Proposition~1.30 and Theorem~1.56 in \cite{BBbook}).
The quasicontinuity of $u$ then follows from 
Proposition~\ref{prop-qcont=>composition-no-PI}.

Finally, assume that every bounded function $u\in\Np(X)$ has a Borel
representative $\ut \in \Np(X)$ and that \ref{it-Borel-qcont} holds
with $U=X$.
Let $u\in\Np(X)$ be bounded and $\ut \in \Np(X)$ be a Borel 
representative of $u$.
By assumption $\ut$ is quasicontinuous.
As $X$ is locally compact, 
Proposition~1.4
in Bj\"orn--Bj\"orn--Shanmugalingam~\cite{BBS5} (or \cite[Proposition~5.27]{BBbook})
and Proposition~4.7 in Bj\"orn--Bj\"orn--Lehrb\"ack~\cite{BBLeh1}
show that also $u$ is quasicontinuous.
Hence \ref{it-bdd-qcont} holds.
\end{proof}

The following open problems are  
natural in view of Corollary~\ref{cor-q-open-Borel-equiv-qcont}
and the comment after it.

\begin{problem}  \label{prob-Borel-repr}
Does every (bounded) $u\in\Np(X)$ have a Borel representative?
\end{problem}

Note that if every bounded $u\in\Np(X)$ has a Borel
representative, then 
also every unbounded
 $v \in \Np(X)$ has a Borel representative.
Indeed, if $u=\arctan v$ has a Borel representative $\ut$ 
(which we may require to have values in $(-\pi/2,\pi/2)$), then
$\tan \ut$ becomes a Borel representative of $v$. 

Since any quasiopen set can be written as a union of a Borel set and a set of 
capacity zero (cf.\ \cite[Lemma~9.5]{BBLat2}
Proposition~\ref{prop-qcont-char}, shows that
the answer to Open problem~\ref{prob-Borel-repr} is 
positive for a particular space $X$
if all Newtonian functions on $X$ are quasicontinuous.

\begin{problem}
Can ``Borel'' in Theorem~\ref{thm-p-path-quasiopen-no-PI-intro} be replaced
by ``measurable''?
Example~\ref{ex-qcont-not-p-path=qopen} shows that it 
cannot be omitted altogether.
\end{problem}

It would follow that a version of 
Corollary~\ref{cor-q-open-Borel-equiv-qcont} 
with ``Borel'' replaced by ``measurable''
would also be possible,
and since all Newtonian functions are measurable, the equivalence of 
\ref{it-bdd-qcont}--\ref{it-Borel-qcont} therein
would follow in that case.

We now proceed to the proof of Theorem~\ref{thm-p-path-quasiopen-no-PI-intro}.
For this, we will need Proposition~\ref{prop-meas-for-anal-Borel}
about measurability of the function $u$
in~\eqref{eq-def-u-from-rho}, which does not rely on any Poincar\'e
inequality.

Recall that a subset of a complete separable metric space 
is \emph{analytic} (or \emph{Suslin}) if it is a continuous image of
a complete separable metric space,
see e.g.\ Kechris~\cite[Definitions~3.1 and 14.1]{Kech}.
By Theorem~14.11 (Suslin's theorem) therein, every Borel subset 
of a complete separable space is analytic. 
(In fact, it shows that 
Borel sets are exactly those analytic sets which are also \emph{coanalytic}, i.e.\
whose complements are analytic.)

Proposition~14.4 in~\cite{Kech} tells us that countable unions and
countable intersections of analytic sets are analytic.
Moreover, if $f:Y\to Z$ is a Borel mapping between two complete separable
metric spaces,
then images and preimages under $f$ of analytic sets are analytic
(also by Proposition~14.4 in~\cite{Kech}).
By Theorem~21.10 in~\cite{Kech} (Luzin's theorem), every analytic subset 
of a complete separable metric 
space $Y$ is $\nu$-measurable for every $\sigma$-finite
Borel measure $\nu$ on $Y$.

When proving Theorem~\ref{thm-p-path-quasiopen-no-PI-intro}
we will  
need Proposition~\ref{prop-meas-for-anal-Borel}
with a lower semicontinuous $\rho$ only, in which case
the proof can be considerably simplified
(in particular the use of Lemma~\ref{lem-JJRRS2.4} can be avoided).
As we find it interesting that Proposition~\ref{prop-meas-for-anal-Borel}
is true also for Borel functions we give a proof of the more
general result, for which we need 
the following result 
from 
J\"arvenp\"a--J\"arvenp\"a--Rogovin--Rogovin--Shan\-mu\-ga\-lin\-gam~\cite[Lemma~2.4]{JJRRS}.

\begin{lemma}  \label{lem-JJRRS2.4}
Let $Z$ be a metric space and let $\YY$ be a class of functions 
$\rho: Z\to[0,\infty]$ such that the following properties 
hold\/\textup{:}
\begin{enumerate}
\item \label{it-cont-Y}
$\YY$ contains all continuous $\rho: Z\to[0,\infty]$\/\textup{;}
\item \label{it-incr-conv-Y}
if $\rho_j\in\YY$ and $\rho_j\nearrow\rho$ then $\rho\in\YY$\/\textup{;}
\item \label{it-lin-comb-Y}
if $\rho,\sigma\in\YY$ and $r,s\in\R^\limplus$ then 
$r\rho+s\sigma\in\YY$\/\textup{;}
\item \label{it-1-rho-Y}
if $\rho\in\YY$ and $0\le\rho\le1$ then $1-\rho\in\YY$.
\end{enumerate}
Then $\YY$ contains all Borel functions $\rho: Z\to[0,\infty]$.
\end{lemma}

\begin{proof}[Proof of Proposition~\ref{prop-meas-for-anal-Borel}]
For $\al>0$, let $\Ll_\al$ consist of all continuous 
curves $\ga: [0,1]\to X$ with $\Lip\gamma\le\alpha$.
(In this proof we want all curves parameterized on the same
interval, and do \emph{not} assume that they are parameterized by
arc length.)
Then $\Ll_\al$ is a 
metric space with respect to the supremum norm.
Since $X$ is complete, it follows
from Ascoli's theorem that $\Ll_\al$ is complete.
(See e.g.\ Royden~\cite[p.\ 169]{royden}
for a version of Ascoli's theorem valid for metric space
valued equicontinuous functions.)

As $X$ is separable, 
it is 
easily verified that $\Ll_\al$ is separable.
Indeed, let $\U$ be a countable base of the topology on $X$
and let $\P$ be the family of all finite sequences 
$Q=((I_1,U_1),\dots, (I_m,U_m))$,
where $I_i\subset [0,1]$ are closed intervals 
with rational endpoints  
and $U_i$ are selected from $\U$. 
Then $\P$ is countable.
For each $Q=((I_1,U_1),\dots (I_m,U_m))\in\P$, 
let $\Ll_{Q}$ be the family of all $\gamma\in\Ll_{\alpha}$ such that 
$\gamma(I_i)\subset U_i$, $i=1,\dots,m$.
Choosing one curve from every nonempty $\Ll_{Q}$ provides us with 
a countable dense system in $\Ll_{\alpha}$.

Let $\YY_\al$ be the collection of all $\rho:X\to[0,\infty]$
for which the functional $\Phi_\rho:\Ll_\al\to[0,\infty]$ defined by
\[
\Phi_\rho:\ga\mapsto \int_\ga\rho\,ds
\]
is Borel.
Lemma~2.2 
in
J\"arvenp\"a--J\"arvenp\"a--Rogovin--Rogovin--Shan\-mu\-ga\-lin\-gam~\cite{JJRRS}
shows that if $\rho$ is continuous then $\rho\in\YY_\al$.
If $\rho_j\in\YY$ and $\rho_j\nearrow\rho$ then 
the monotone convergence theorem implies that for all $\ga\in\Ll_\al$,
\[
\Phi_{\rho_j}(\ga) = \int_\ga \rho_j\,ds \nearrow \int_\ga \rho\,ds
= \Phi_{\rho}(\ga),
\]
i.e.\ that $\Phi_{\rho}$ is a limit of Borel functions on $\Ll_\al$, and
hence Borel.
Thus, $\rho\in\YY_\al$.
In particular, \ref{it-cont-Y} and \ref{it-incr-conv-Y} in
Lemma~\ref{lem-JJRRS2.4} are satisfied by $\YY_\al$.
The properties \ref{it-lin-comb-Y} and \ref{it-1-rho-Y} therein
follow from the linearity of the integral.

We can thus conclude from Lemma~\ref{lem-JJRRS2.4} that for every $\al>0$
and every Borel $\rho:X\to[0,\infty]$, the functional 
$\Phi_\rho:\Ll_\al\to[0,\infty]$ is Borel.
Let $\rho:X\to[0,\infty]$ be a fixed Borel function
and $\alpha>0$ be arbitrary.
We shall prove that the set
$G=\{x\in X:u_\rho(x)<\alpha\}$ is measurable. 

First, assume that $\rho\ge\de$ for some $\de>0$. 
Note that if $x\in G$, then for every $\ga\in\Ga_x$ with 
$\int_\ga \rho\,ds<\al$, we have $l_\ga<\al/\de$ and hence the reparameterized
curve $\gat(t):=\ga(l_\ga t)$ belongs to $\Ll_{\al/\de}$.
It follows that for $x\in G$, the infimum in the definition of $u_\rho$ can
equivalently be taken over all $\ga\in\Ga_x\cap\Ll_{\al/\de}$.
We define
\[
 \Gamma_1=\biggl\{\ga\in\Ll_{\al/\de}:\int_{\ga}\rho\,ds<\alpha\biggr\}
\quad \text{and} \quad
\Gamma_2=\{\ga\in\Ll_{\al/\de}: \ga(1)\notin U\}.
\]
Also let 
$f_j: \Ll_{\al/\de} \to X$ be the evaluation maps given by 
$f_j(\ga)=\ga(j)$, $j=0,1$,
which are clearly $1$-Lipschitz,
 and thus Borel.

By the above, the functional $\Phi_\rho:\Ll_{\al/\de}\to[0,\infty]$ is Borel, 
and thus $\Ga_1=\Phi_\rho^{-1}([0,\alp))$ is a Borel subset of $\Ll_{\al/\de}$.
The set $\Ga_2$ is the preimage of the analytic set $X\setminus U$ under 
the Borel mapping $f_1$, and thus $\Ga_2$ is analytic by Proposition~14.14
in Kechris~\cite{Kech}.
It thus follows from Proposition~14.4
in~\cite{Kech} that $\Ga_1\cap\Ga_2$ is analytic.
Since $G=f_0(\Ga_1\cap\Ga_2)$, we conclude that $G$ is analytic, and thus 
measurable, by Luzin's theorem~\cite[Theorem~21.10]{Kech}.

Now, let $\rho:X\to[0,\infty]$ be arbitrary and set $\rho_j=\rho+1/j$,
$j=1,2,\ldots$.
By the above, each $u_{\rho_j}$ is measurable.
We shall show that $u_\rho=\lim_{j\to\infty} u_{\rho_j}$, which implies the
measurability of $u_\rho$.

Given $x\in X$ and $\eps>0$, there exists $\ga\in\Ga_x$ such that
$\int_\ga\rho\,ds < u_\rho(x)+\eps$.
Since $\ga$ is rectifiable, we have 
\[
u_\rho(x) \le u_{\rho_j}(x) \le \int_\ga \rho_j\,ds \le \int_\ga (\rho+1/j)\,ds
< u_\rho(x)+\eps+\frac{l_\ga}{j}.
\]
Letting $j\to\infty$ and then $\eps\to0$ shows that 
$u_\rho(x)=\lim_{j\to\infty} u_{\rho_j}(x)$ for all $x\in X$,
and  we conclude that $u_\rho$ is measurable.
\end{proof}

We are now ready to prove Theorem~\ref{thm-p-path-quasiopen-no-PI-intro}.
It will be obtained in a more general form, using coanalytic sets.
For separable $X$ we have the following result.

\begin{thm} \label{thm-p-path-quasiopen-no-PI-separable}
Assume that $X$ is complete and separable, and that every 
bounded $u \in \Np(X)$ is quasicontinuous.
Then every coanalytic
\p-path open set $U\subset X$ is quasiopen.
\end{thm}

For nonseparable spaces we use the fact that $\supp \mu$ is always 
separable, by Proposition~1.6 in \cite{BBbook}.
This way we can avoid nonseparable analytic (Suslin) sets and reduce our
considerations to separable spaces where Suslin and analytic sets are the
same, cf.\ Hansell~\cite{Hans} and Kechris~\cite{Kech}.
The following result is primarily designed for nonseparable spaces.
However, it improves the criterion also for separable spaces, 
since we do not impose any assumptions on $U\setm\spt\mu$.

\begin{thm} \label{thm-p-path-quasiopen-no-PI-gen}
Assume that $\supp \mu$ is complete, and that every 
bounded $u \in \Np(X)$ is quasicontinuous.
Then every 
\p-path open set $U\subset X$,  such that $\supp \mu \setm U$ is 
analytic\/ \textup{(}in $\spt\mu$\textup{)}, is quasiopen.
\end{thm}

Theorem~\ref{thm-p-path-quasiopen-no-PI-separable}
is a special case of Theorem~\ref{thm-p-path-quasiopen-no-PI-gen},
but we will prove Theorem~\ref{thm-p-path-quasiopen-no-PI-separable}
first and then use it to prove 
Theorem~\ref{thm-p-path-quasiopen-no-PI-gen}.

\begin{remark}
By Theorem~1.1 in
Bj\"orn--Bj\"orn--Shanmugalingam~\cite{BBS5} (or \cite[Theorem~5.29]{BBbook})
and the comments after Proposition~4.7 in 
Bj\"orn--Bj\"orn--Lehrb\"ack~\cite{BBLeh1},
the assumptions in Theorem~\ref{thm-p-path-quasiopen-no-PI-gen} 
hold in particular if $X$ 
(or more generally $\supp \mu$)  is complete and locally
compact, and continuous functions are dense in $\Np(X)$ or equivalently
in $\Np(\supp \mu)$. 
The last equivalence follows from Lemma~5.19 in  \cite{BBbook},
which also implies that
every (bounded) function in $\Np(X)$ is quasicontinuous
if and only if every (bounded) function in $\Np(\supp \mu)$ is 
quasicontinuous.

At this point we would also like to mention that it follows from
the recent results in Ambrosio--Colombo--Di Marino~\cite{AmbCD}
and Ambrosio--Gigli--Savar\'e~\cite{AmbGS} that, if $\supp \mu$ is a complete
doubling metric space and $1<p<\infty$, then
Lipschitz functions are dense in $\Np(X)$.
Thus, the assumptions on $X$ in Theorem~\ref{thm-p-path-quasiopen-no-PI-gen}
hold if $X$ (or $\supp \mu$) is a complete
doubling metric space and $1<p<\infty$.
It is easily verified that such a space is automatically locally compact. 
\end{remark}

\begin{proof}[Proof of Theorem~\ref{thm-p-path-quasiopen-no-PI-separable}]
We follow the proof of Theorem~\ref{thm-p-path-quasiopen}.
Let $\rho\in L^p(X)$ be as therein.
The Vitali--Carath\'eodory theorem (Proposition~7.14 in Folland~\cite{Foll})
provides us with a lower semicontinuous pointwise majorant of $\rho$ which
also belongs to $L^p(X)$.
We can therefore without loss of generality assume that $\rho$ is
lower semicontinuous.

The measurability of the function $u$
in~\eqref{eq-def-u-from-rho} is now 
guaranteed by Proposition~\ref{prop-meas-for-anal-Borel} 
and the assumption that $U$ is coanalytic, 
rather than by the \p-Poincar\'e inequality and Theorem~1.11
in~\cite{JJRRS}.
Also, since quasicontinuity of bounded Newtonian functions is assumed, it need not be
concluded from the \p-Poincar\'e inequality.
The rest of the proof goes through verbatim.
\end{proof}

\begin{proof}[Proof of  Theorem~\ref{thm-p-path-quasiopen-no-PI-gen}]
Let $Y=\supp \mu$.
It follows from Proposition~1.6 in \cite{BBbook} that $Y$ is separable.
Let $U \subset X$ be a \p-path open set 
such that $Y \setm U$ is analytic.
It follows from Proposition~1.53 in \cite{BBbook} that \p-almost
no curve intersects
$X \setm Y$, and thus $U \cap Y$ is also \p-path open.

Hence, as all bounded $u\in\Np(Y)$ are quasicontinuous 
by Lemma~5.19 in~\cite{BBbook},
Theorem~\ref{thm-p-path-quasiopen-no-PI-separable} implies that
$U \cap Y$ is quasiopen in $Y$, i.e.\
for every $\eps >0$ there is a relatively open set $G \subset Y$ such that
$U \cap Y \subset G$ and $\CpY(G \setm U) < \eps$.
Thus 
there exists $v\in\Np(Y)$ such that $v\ge \chi_{G\setm U}$
and $\|v\|_{\Np(Y)}^{p}<\eps$.
Since \p-almost no curve intersects $X\setm Y$
and $\mu(X \setm Y)=0$, extending $v$ by 1 to $X\setm Y$
shows that $\Cp(G \setm U)\le \|v\|_{\Np(X)}^p= \|v\|_{\Np(Y)}^{p}< \eps$.

Let $G' = G \cup (X \setm Y)$, which is open and contains $U$.
By Proposition~1.53 in \cite{BBbook}  we see that
$\Cp(X \setm Y)=0$ and hence
\[
\Cp(G' \setm U) \le \Cp(G \setm U) + \Cp(X \setm Y) = \Cp(G \setm U) < \eps,
\]
showing that 
$U$ is quasiopen in $X$.
\end{proof}

\begin{proof}[Proof of  Theorem~\ref{thm-p-path-quasiopen-no-PI-intro}]
This is a special case of Theorem~\ref{thm-p-path-quasiopen-no-PI-gen}.
\end{proof}

\end{document}